\documentclass[a4paper, 12pt]{article}
\usepackage[utf8]{inputenc}

\usepackage[top=2.5cm, bottom=2.5cm, left=2.5cm, right=2.5cm]{geometry}

\usepackage{amsmath, amssymb, mathtools}
\usepackage[hidelinks]{hyperref}

 \usepackage{xcolor}

\usepackage{amsthm}
\usepackage[noabbrev]{cleveref}
\newtheorem{thm}{Theorem}[section]
\newtheorem{lm}[thm]{Lemma}
\newtheorem{res}[thm]{Result}
\newtheorem{crl}[thm]{Corollary}

\theoremstyle{definition}

\newtheorem{rmk}[thm]{Remark}
\newtheorem{df}[thm]{Definition}
\newtheorem{nota}[thm]{Notation}

\newcommand{\RR}{\mathbb R}
\newcommand{\FF}{\mathbb F}

\newcommand{\ZZ}{\mathbb Z}
\newcommand{\NN}{\mathbb N}

\newcommand{\sett}[2]{ \left\{ #1 \, \, || \, \, #2 \right \} }

\newcommand{\floor}[1]{\left \lfloor #1 \right \rfloor}
\newcommand{\ceil}[1]{\left \lceil #1 \right \rceil}

\newcommand{\mf}{\mathcal F}

\newcommand{\ms}{\mathcal S}

 \DeclareMathOperator{\PG}{PG}
    \newcommand{\pg}{\PG}
 \DeclareMathOperator{\AG}{AG}
    \newcommand{\ag}{\AG}
 \DeclareMathOperator{\PGL}{PGL}

 \DeclareMathOperator{\pr}{pr}
 \newcommand{\fpi}{\FF_p \cup \{\infty\}}

\makeatletter
\let\@fnsymbol\@arabic  
\makeatother 

\title{Points below a parabola in affine planes of prime order}
\author{
 Sam Adriaensen
  \thanks{Department of Mathematics and Data Science, Vrije Universiteit Brussel, Pleinlaan 2, 1050 Elsene, Belgium. \href{mailto:sam.adriaensen@vub.be}{sam.adriaensen@vub.be}} 
  \thanks{Department of Mathematical Sciences, Worcester Polytechnic Institute, 100 Institute Road, 01605 Worcester, MA, US.} 
 \and Zsuzsa Weiner
  \thanks{HUN-REN-ELTE Geometric and Algebraic Combinatorics Research Group, P\'azm\'any P\'eter s\'et\'any 1/C, H-1117 Budapest, Hungary.
   \href{mailto:zsuzsa.weiner@gmail.com}{zsuzsa.weiner@gmail.com}} 
}
\date{}

\begin{document}

\maketitle

\begin{abstract}
 The elements of a finite field of prime order canonically correspond to the integers in an interval.
 This induces an ordering on the elements of the field.
 Using this ordering, Kiss and Somlai recently proved interesting properties of the set of points below the diagonal line.
 In this paper, we investigate the set of points lying below a parabola.
 We prove that in some sense, this set of points looks the same from all but two directions, despite having only one non-trivial automorphism.
 In addition, we study the sizes of these sets, and their intersection numbers with respect to lines.
\end{abstract}



\section{Introduction}
Let $q$ be a prime power, let $\FF_q$ be the finite field of order $q$, and let $\ag(2,q)$ denote the Desarguesian affine plane of order $q$.
Let $\ell_\infty$ denote the line that completes $\ag(2,q)$ to a projective plane.
The points of $\ell_\infty$ are often called \emph{directions}, as they correspond to the slopes of the lines of $\ag(2,q)$.
We will denote these directions as $(d)$, with $d \in \FF_q \cup \{\infty\}$.
The affine lines with slope $d$ are exactly the lines with equation $aY = mX + b$ with $\frac m a = d$.

A set $\ms$ of points in $\ag(2,q)$ is said to \emph{determine} direction $(d)$ if $\ms$ contains 2 points spanning a line with slope $d$.
Sets determining few directions have been thoroughly investigated, see e.g.\ the pioneering work of Rédei \cite{Redei:70} and the seminal work of Blokhuis, Ball, Brouwer, Storme, and Sz\H onyi \cite{BBBSS:99,Ball:03}.
The authors essentially proved that a set of $q$ points determining few directions must be the translate of a set which is linear over a subfield of $\FF_q$.

By the pigeonhole principle, any set with more than $q$ points in $\ag(2,q)$ determines all directions.
However, we can associate sets of directions to larger sets of points in a sensible way.
Given a set $\ms$ of points in $\ag(2,q)$, we call a direction $(d)$ \emph{equidistributed} if all lines with slope $d$ intersect $\ms$ in $\floor{|\ms|/q}$ or $\ceil{|\ms|/q}$ points.
Otherwise, we call $(d)$ a \emph{special direction}.
This idea was introduced by Ghidelli \cite{Ghidelli:20}, who extended the results of Rédei \cite{Redei:70} and Sz\H onyi \cite{Szonyi:96} in case $q$ is prime.

In case that $|\ms|$ is divisible by $q$, a direction $(d)$ is equidistributed if and only if all lines with direction $(d)$ intersect $\ms$ in the same number of points.
This special case was further investigated by Kiss and Somlai \cite{KS} in planes of prime order.
They constructed a set of $\binom p2$ points in $\ag(2,p)$ with only 3 special directions, which is not a union of lines, for every prime number $p > 2$.
This stands in sharp contrast to the fact that any set of $p$ points in $\ag(2,p)$ is either a line, or determines at least $\frac{p+3}2$ directions.

The construction of Kiss and Somlai \cite{KS} works as follows.
In case $p$ is prime, every element of $\FF_p$ uniquely corresponds to an integer in the interval $[0, p-1]$.
This naturally induces an order relation $<$ on $\FF_p$.

\begin{res}[{\cite[Theorem 1.1]{KS}}]
 \label{Res:KissSomlai}
 Let $p > 2$ be prime.
 Consider the set of $\frac{p-1}2 p$ points
 \[
  \ms = \sett{ (x,y) \in \FF_p^2 }{ y < x }
 \]
 in $\ag(2,p)$.
 This set has exactly 3 special directions, namely $(0)$, $(1)$, and $(\infty)$.
 Moreover, every set of $\ag(2,p)$ having exactly 3 special directions is equivalent up to affine transformation to either $\ms$ or its complement.
\end{res}

In this paper, first we will investigate the set of points below a parabola, that is 
\[
{\mathcal{S}} = \sett{(x,y) \in \FF_p^2}{ y < \alpha x^2 + \beta x + \gamma},
\]
where $\alpha \in \FF_p^*, \beta, \gamma \in \FF_p$.
In Theorem \ref{Thm:SizeS}, we estimate the spectrum of the possible sizes of such sets.
To be able to express the unexpected properties of such sets $\ms$, we introduce the \emph{projection function}, following \cite{KS}.

\begin{df}
We define the \emph{projection function} of $\ms$ from direction $(d) \in \fpi$ to be $\pr_{\ms,d}:\FF_p \to \NN$ that maps $b$ to the number of points of $\ms$ on the line $Y=dX+b$ if $(d) \neq (\infty)$, or the line $X+b=0$ if $(d) = (\infty)$.
\end{df}

We will see that two directions, namely $(0)$ and $(\infty)$, are distinguished from all other directions.
For the remaining directions $(d)$, the projection functions are cyclic shifts of each other, see Theorem \ref{Thm:Parabola}.
It is natural to think that such a phenomenon is due to the fact that $\mathcal S$ has large symmetry, but this is disproved in Theorem \ref{thm:sym}.
Finally, we also study the point set above a parabola and also these point sets by adding the parabola itself and we prove similar results as before. 
In order to prove our results, we use results concerning the quadratic character on $\FF_p$.

\bigskip

In an accompanying paper, the authors of this note together with Tam\'as Sz\H onyi take another approach towards generalising \Cref{Res:KissSomlai}, by exploiting a connection with codes generated by incidence matrices of finite projective planes \cite{Directions}.

\section{The distribution of squares modulo a prime}

Let $p>2$ be a prime number.
Throughout this section, $Q$ and $N$ denote the sets of non-zero squares and non-squares of $\FF_p$ respectively.
The \emph{quadratic character} on $\FF_p$ is defined as the function
\[
 \chi: \FF_p \to \RR: x \mapsto \begin{cases}
  1 & \text{if } x \in Q, \\
  0 & \text{if } x = 0, \\
  -1 & \text{if } x \in N.
 \end{cases}
\]

We can naturally interpret $\chi$ as a function $\ZZ \to \RR$.
The sum of the quadratic character over an interval is bounded by a famous result of P\'olya \cite{Polya:18} and Vinogradov \cite{Vinogradov:18}.

\begin{res}[P\'olya-Vinogradov inequality]
 \label{Res:PolyaVinogradov}
 For any two integers $a, b$,
 \[
  \left| \sum_{x=a}^b \chi(x) \right| \leq \sqrt p \ln(p).
 \]
\end{res}

As a converse to the P\'olya-Vinogradov inequality, we have the following result by S\'arközy \cite{Sarkozy:77} and Sokolovski\u{\i} \cite{Sokolovskii:82}, see also Csikv\'ari \cite{Csikvari:09}.

\begin{res}
 \label{Res:Csikvari}
 There exists an element $b \in \FF_p$ such that
 \[
  \left| \sum_{x=1}^b \chi(x) \right| \geq \frac1{2\pi} \sqrt p.
 \]
\end{res}

\begin{nota}
 We let $\nu:\FF_p \to \NN$ denote the function that maps an element $x \in \FF_p$ to the corresponding integer in $\{0,\dots,p-1\}$.
\end{nota}

By a result of Lebesgue \cite{Lebesgue:42}, see also Aebi and Cairns \cite{Aebi:Cairns}, we know the following about the sum of the squares of $\FF_p$.

\begin{res}
 \label{Res:Lebesgue}
 Let $p$ be an odd prime, and let $n = \left| N \cap \left[ 1, \frac{p-1}2 \right] \right|$.
 Then
 \[
  \sum_{x \in Q} \nu(x) = \begin{cases}
   p\frac{p-1}{4} & \text{if } p \equiv 1 \pmod 4, \\
   pn & \text{if } p \equiv -1 \pmod 8, \\
   p \left( \frac n3 + \frac{p-1}6 \right) & \text{if } p \equiv 3 \pmod 8.
  \end{cases}
 \]
\end{res}

By considering the P\'olya-Vinogradov inequality on the interval $\left[1,\frac{p-1}2\right]$, we find the following bound on $n$.

\begin{crl}
 \label{Crl:Lebesgue}
 Let $n$ be as in \Cref{Res:Lebesgue}.
 Then $\left| n - \frac{p-1}4 \right| \leq \frac 1 2 \sqrt p \ln(p)$ and hence
 \[
  \left|\sum_{x \in Q} \nu(x) - p\frac{p-1}4 \right| \leq \begin{cases}
   0 & \text{if } p \equiv 1 \pmod 4, \\
   \frac{1}2 p\sqrt{p}ln(p) & \text{if } p \equiv -1 \pmod 8, \\
   \frac{1}6 p\sqrt{p}ln(p) & \text{if } p \equiv 3 \pmod 8.
  \end{cases}
 \]
\end{crl}

Call two elements $x$ and $y$ of $\FF_p$ \emph{consecutive} if $y = x+1 \neq 0$.

\begin{lm}
 \label{Lm:Consecutive}
 Let $f$ be a quadratic polynomial on $\FF_p$.
 The image of $f$ contains at least $\frac{p-5}{4}$ pairs of consecutive elements.
\end{lm}

\begin{proof}
 We can write $f(X) = \alpha (X-x)^2 + \gamma$ for certain $\alpha, \gamma, x \in \FF_p$.
 By a result of Jacobsthal \cite{Jacobsthal:06}, $Q$ has $\floor{\frac{p-3}4}$ pairs of consecutive elements, hence $Q \cup \{0\}$ has $\floor{\frac{p+1}4}$ such pairs, and $N$ has $\floor{\frac{p-1}4}$ pairs of consecutive elements.
 The lemma follows from the fact that the image of $f$ equals $(Q \cup \{0\}) + \gamma$ \footnote{Given a set $S \subseteq \FF_p$ and an element $t \in \FF_p$, $S + t$ denotes the set $\sett{s+t}{s \in S}$.} if $\alpha \in Q$, and $(N \cup \{0\}) + \gamma$ if $\alpha \in N$, and that the translation of a set destroys at most one pair of consecutive elements.
\end{proof}

\section{The set of points below a parabola.}

\begin{thm}
 \label{Thm:Parabola}
 Take a quadratic polynomial $f(X) = \alpha X^2 + \beta X + \gamma$, $\alpha \neq 0$ over $\FF_p$, with $p > 2$ prime.
 Consider the set
 \[
  \ms = \sett{(x,y) \in \FF_p^2}{ y < f(x)}
 \]
 of points below the parabola $Y = f(X)$.
 For any $d \in \FF_p^*$ and $b \in \FF_p$,
 \[
  \pr_{\ms,d}(b) = \pr_{\ms,1}\left( b - \left(\beta-\frac{d+1}2\right)\frac{d-1}{2\alpha} \right).
 \]
 Moreover, the image of $\pr_{\ms,d}$ is an interval whose length lies between $\frac{\sqrt p}{2 \pi}$ and $\sqrt p \ln(p)$.
\end{thm}

\begin{proof}
 We need to prove that for any value of $b$,
 \[
  \pr_{\ms,d}\left( b + \left(\beta-\frac{d+1}2\right)\frac{d-1}{2\alpha} \right)
 \]
 is independent of the choice of $d \in \FF_p^*$.
 This integer equals the number of solutions to the inequality
 \[
  dX + b + \left(\beta-\frac{d+1}2\right)\frac{d-1}{2\alpha} < \alpha X^2 + \beta X + \gamma.
 \]
 We apply the linear transformation $X' = X - \frac{d-1}{2\alpha}$ to see that we may equivalently count the number of solutions to the inequality
 \begin{equation}
  \label{Eq:X'}
  X' + b + \delta_d(X') < \alpha X'^2 + \beta X' + \gamma + \delta_d(X'),
 \end{equation}
 with $\delta_d(X') = (d-1) X' + \frac{(d-1)^2}{4\alpha} + \beta \frac{d-1}{2\alpha}$.
 
 Let $N_{d,b}$ denote the set of solutions to Equation (\ref{Eq:X'}).
 Then we need to prove that $|N_{d,b}|$ is independent of the choice of $d \in \FF_p^*$.
 Since we know that $\sum_{b \in \FF_p} |N_{d,b}| = |\ms|$ for any $d$, it actually suffices to prove that for any $b$,
 \[
  |N_{d,b+1}|-|N_{d,b}| = |N_{d,b+1} \setminus N_{d,b}| - |N_{d,b} \setminus N_{d,b+1}|
 \]
 is independent of $d \in \FF_p^*$.
 Note that
 \begin{align*}
  x \in N_{d,b+1} \setminus N_{d,b} && \iff && x+b+\delta_d(x)=-1 && \text{ and } && f(x) + \delta_d(x) \neq 0, \\
  x \in N_{d,b} \setminus N_{d,b+1} && \iff && x+b+1=f(x) && \text{ and } && f(x) + \delta_d(x) \neq 0.
 \end{align*}
 Let $k$ denote the number of solutions to $X+b+1 = f(X)$.
 Since $d \neq 0$, there is a unique solution $x_0$ to $X + b + \delta_d(X) = -1$.
 
 First suppose that $x_0$ is a solution to $f(X) + \delta_d(X) = 0$.
 Then it is also a solution to $X+b+1 = f(X)$, hence
 \begin{align*}
  |N_{d,b+1} \setminus N_{d,b}| = 0, &&
  |N_{d,b} \setminus N_{d,b+1}| = k-1.
 \end{align*}

 If $x_0$ is not a solution to $f(X) + \delta_d(X) = 0$, then
 \begin{align*}
  |N_{d,b+1} \setminus N_{d,b}| = 1, &&
  |N_{d,b} \setminus N_{d,b+1}| = k.
 \end{align*}
 To see that the latter equality holds, we need to prove that no solution $x$ of $X+b+1 = f(X)$ satisfies $f(X) + \delta_b(X) = 0$.
 If that were the case, then $x + b + \delta_d(x) = -1$, which implies that $x = x_0$ and contradicts that $f(x_0) + \delta_d(x_0) \neq 0$.

 We conclude that in both cases
 \[
  |N_{d,b+1} \setminus N_{d,b}| - |N_{d,b} \setminus N_{d,b+1}| = 1-k,
 \]
 which is indeed independent of the choice of $d \in \FF_p^*$.

 Note that $k$ is the number of roots of the quadratic equation $f(X) - X - b - 1 = 0$.
 Therefore, $k$ depends on the quadratic character of the discriminant $(\beta-1)^2 + 4\alpha(b+1-\gamma)$.
 More specifically,
 \[
  1-k = -\chi( (\beta-1)^2 + 4\alpha(b+1-\gamma) ).
 \]
 This means $N_{d,b}$ and $N_{d,b+1}$ differ by $1-k$ and $|1-k| \leq 1$, hence that the image of $\pr_{\ms,d}$ consists of all integers in an interval.
 Moreover, for any $a,b \in \FF_p$,
 \begin{align*}
  \Big | |N_{d,b}| - |N_{d,a}| \Big |
  &= \left | \sum_{x = a+1}^b \chi( (\beta-1)^2 + 4\alpha(x-\gamma) ) \right| \\
  &= |\chi(4 \alpha)| \left | \sum_{x = a+1}^b \chi \left( \frac{(\beta-1)^2}{4\alpha} + x-\gamma \right) \right|
  = \left | \sum_{x = \frac{(\beta-1)^2}{4\alpha} + a+1-\gamma}^{\frac{(\beta-1)^2}{4\alpha} + b-\gamma} \chi \left( x \right) \right|
 \end{align*}
 Therefore, the length of the interval formed by the image of $\pr_{\ms,d}$ equals the maximum value of the above sum.
 We know from Results \ref{Res:PolyaVinogradov} and \ref{Res:Csikvari} that this maximum value is between $\frac{\sqrt p}{2 \pi}$ and $\sqrt p \ln(p)$.
\end{proof}

From now on, we let $f(X)$ and $\ms$ be as in \Cref{Thm:Parabola}.

\begin{rmk}
 \label{Rmk:Parabola}
 If $f_1(X)$ can be transformed in $f_2(X)$ by a linear transformation in the variable $X$, i.e.\ $f_1(X) = f_2(aX + b)$ for some $a\neq 0$ and $b$ in $\FF_p$, then the sets of points defined by $y < f_1(x)$ and $y < f_2(x)$ are isomorphic.
 In particular, we can choose $\beta$ to be $0$ without loss of generality, and we will do so for the remainder of this section.
 Moreover, only the quadratic character $\chi(\alpha)$ of $\alpha$ matters.
\end{rmk}

Next, we prove bounds on the size of $\ms$.

\begin{thm}
 \label{Thm:SizeS}
 Let $\ms$ be the set of points below a parabola in $\ag(2,p)$, with $p > 2$ prime.
 Then
 \begin{align*}
  \left||\ms| - \binom p 2\right| \leq c_p p \sqrt p \ln(p) && \text{with} &&
  c_p = \begin{cases}
   1 & \text{if } p \equiv 1 \pmod 4, \\
   2 & \text{if } p \equiv -1 \pmod 8, \\
   \frac43 & \text{if } p \equiv 3 \pmod 8.
  \end{cases}
 \end{align*}
 Moreover, $|\ms|$ is a multiple of $p$.
\end{thm}

\begin{proof}
 Without loss of generality, we may suppose that $\beta = 0$.
 For every $x \in \FF_p$, the points of $\ms$ with first coordinate equal to $x$ are 
 \[
  (x,0),\, (x,1),\, \dots,\, (x,f(x)-1).
 \]
 Hence, there are $\nu(f(x))$ points of $\ms$ with first coordinate equal to $x$.
 This yields that
 \[
  |\ms| = \sum_{x \in \FF_p} \nu(\alpha x^2 + \gamma).
 \]
 Note that over $\FF_p$,
 \[
  \sum_{x \in \FF_p} \alpha x^2 + \gamma = 0,
 \]
 hence $|\ms|$ is divisible by $p$.
 
 First consider the case where $\alpha \in Q$.
 Then we find that
 \[
  |\ms| = \sum_{x \in \FF_p} \nu(x^2 + \gamma) = 2 \sum_{y\in Q} \nu(y) + p \Big( \nu(\gamma) - 2 |\sett{y \in Q}{y \geq p-\gamma}| \Big). 
 \]
 Note that
 \[
  \nu(\gamma) - 2 |\sett{y \in Q}{y \geq p-\gamma}| = - \sum_{y=\nu(-\gamma)}^{p-1} \chi(y),
 \]
 whose absolute value is bounded by $\sqrt p \ln(p)$ by the P\'olya-Vinogradov inequality.
 Moreover,
 \[
  \left| 2 \sum_{y \in Q} \nu(y) - \binom p 2 \right| \leq c_p - 1,
 \]
 by \Cref{Crl:Lebesgue}, which completes the proof for $\alpha \in Q$.

 The case where $\alpha \in N$ is completely analogous, using that
 \[
  \sum_{y \in N} \nu(y) = \binom p 2 - \sum_{y \in Q} \nu(y). \qedhere
 \]
\end{proof}

\begin{rmk}
 Note that in \Cref{Thm:SizeS}, if $\gamma = 0$, then we can get
 \[
  \left||\ms| - \binom p 2\right| \leq (c_p-1) p \sqrt p \ln(p).
 \]
 In particular, if $p \equiv 1 \pmod 4$, then $|\ms| = \binom p 2$.
\end{rmk}

Next, we examine the symmetries of the set $\ms$ of points below the parabola $Y = f(X) = \alpha X^2 + \gamma$.
We prove that it has a unique non-trivial symmetry, namely a reflection around the vertical axis, provided that $p$ is large enough.

\begin{thm}
\label{thm:sym}
 Suppose that $p \geq 97$.
 Embed $\ag(2,p)$ into $\pg(2,p)$ and consider $\ms$ as a point set in $\pg(2,p)$.
 The only non-trivial element of $\PGL(3,p)$ that stabilises $\ms$ setwise is $(X,Y,Z) \mapsto (-X,Y,Z)$.
\end{thm}

\begin{proof}
 First we prove that we can distinguish the points $(0)$ and $(\infty)$ from all other points of $\pg(2,q)$.
 Take any other point $P$.
 If $P$ lies on the line at infinity $\ell_\infty$, then $P = (d)$ for some $d \in \FF_p^*$.
 \Cref{Thm:Parabola} implies that the lines through $P$ different from $\ell_\infty$ can intersect $\ms$ in at most $\sqrt p \ln p + 1$ different numbers of points.
 If $P$ does not lie on $\ell_\infty$, then any line $\ell$ through $P$ but not through $(0)$ or $(\infty)$ can intersect $\ms$ in at most $\sqrt p \ln p + 1$ different numbers of points.
 We conclude that the lines through $P$ intersect $\ms$ in at most $\sqrt p \ln p + 3$ different numbers of points.

 As noted in the beginning of the proof of \Cref{Thm:SizeS}, the line with equation $X=x$ intersects $\ms$ in $\nu(f(x))$ points.
 Since the image of $f$ has size $\frac{p+1}2$, the lines through $(\infty)$ intersect $\ms$ in at least $\frac{p+1}2$ different numbers of points.
 Moreover, by \Cref{Lm:Consecutive}, the image of $f$ takes at least $\frac{p-5}4$ consecutive values of $\FF_p$.

 Now consider the lines through $(0)$.
 If a point $(x,y)$ belongs to $\ms$, then $(x,y+1)$ also belongs to $\ms$, except if $y+1 = f(x)$.
 Since $f$ takes $\frac{p+1}{2}$ different values, the lines through $(0)$ intersect $\ms$ in at least $\frac{p+1}2$ different numbers of points.
 Note also that since $f$ only takes the value $\gamma$ exactly once, these numbers contain at most one pair of consecutive integers.

 Since $p \geq 97$, $\sqrt p \ln (p) + 3 < \frac{p+1}2$, hence the points $(0)$ and $(\infty)$ are distinguished from all other points.
 In addition, we can distinguish $(0)$ and $(\infty)$ from each other by looking at how many pairs of consecutive integers occur as intersection numbers of the lines through $(0)$ and $(\infty)$.
 In particular, this means that any projectivity $\varphi$ stabilising $\ms$ must fix $(0)$ and $(\infty)$, and hence stabilise $\ell_\infty$.
 
 Next we prove that such a projectivity $\varphi$ must fix all lines through $(0)$.
 It follows from our previous arguments that if $y+1$ and $y+2$ both belong to the image of $f$, with $y \neq -1, -2$, then $Y=y$, $Y=y+1$, and $Y=y+2$ all intersect $\ms$ in a different number of points.
 In particular, no other line than $Y=y+1$ intersect $\ms$ in the same number of points as $Y=y+1$.
 Since the image of $f$ contains many consecutive elements by \Cref{Lm:Consecutive}, there are many integers $k$ such that there exists a unique line through $(0)$ intersecting $\ms$ in exactly $k$ points.
 It follows that $\varphi$ fixes at least 3 lines through $(0)$, hence it must fix all lines through $(0)$.

 We have established that a projectivity $\varphi$ stabilising $\ms$ must fix $(0)$, $(\infty)$, and all lines through $(0)$.
 It follows that $\varphi(X,Y,Z) = (aX+bZ,Y,Z)$ for some $a,b \in \FF_p$.
 The line $X=0$ is the only line through $(\infty)$, distinct from $\ell_\infty$, containing $\nu(\gamma)$ points.
 Therefore, $\varphi$ should fix this line, which implies that $b=0$.
 Furthermore, for any $x\in \FF_p^*$, the only lines through $(\infty)$, distinct from $\ell_\infty$, containing $\nu(f(x))$ points are $X=x$ and $X=-x$.
 It follows that $a$ equals $1$ or $-1$.

 Lastly, it is easy to check that $\varphi(X,Y,Z) = (-X,Y,Z)$ actually stabilises $\ms$.
\end{proof}

Now that we determined the symmetries of the sets $\ms$ under consideration, we might also wonder how many sets $\ms$ we found up to isomorphism.
To address this, we use the following lemma.

\begin{lm}
 \label{Lm:Equiv Ineq}
 Let $a, b \in \FF_p$.
 Then $a < b$ if and only if $-a-1 > -b-1$.
\end{lm}

\begin{proof}
 First suppose that $a < b$, which we formulate denote as $\nu(a) < \nu(b)$.
 Since
 \[
  \nu(-a) = \begin{cases}
   \nu(a) = 0 & \text{if } a = 0, \\
   p-\nu(a) & \text{if } a \neq 0,
  \end{cases}
 \]
 We see that $\nu(a) < \nu(b)$ is equivalent to $\nu(-a) > \nu(-b)$ if $a$ and $b$ are both non-zero, and equivalent to $a = 0 \neq b$ otherwise.

 Now consider the inequality $\nu(-a-1) > \nu(-b-1)$.
 Since
 \[
  \nu(-a-1) = \begin{cases}
   p - 1 - \nu(a) = p-1 & \text{if } a = 0, \\
   \nu(-a) - 1 & \text{if } a \neq 0, 
  \end{cases}
 \]
 We see that $\nu(-a-1) > \nu(-b-1)$ if $a$ and $b$ are both non-zero, and equivalent to $a=0 \neq b$ otherwise.
\end{proof}

\begin{rmk}
 We saw that any quadratic polynomial $f(X)$ gives rise to a set $\ms = \sett{(x,y) \in \FF_p^2}{y < f(x)}$ where the projection functions $\pr_{\ms,d}$ with $d \in \FF_p^*$ are cyclic shifts of each other.
 We might wonder how many such sets we find up to isomorphism, where isomorphism is understood to be an affine transformation.
 In the condition $y < f(x)$, we might replace $<$ by $\leq$, $>$, or $\geq$.
 
 There are two things to note.
 First, by \Cref{Lm:Equiv Ineq}, the affine transformation $(X,Y) \mapsto (X,-Y-1)$ maps $\sett{(x,y) \in \FF_p^2}{y < f(x)}$ to $\sett{(x,y) \in \FF_p^2}{y > -f(x)-1}$ and maps the set $\sett{(x,y) \in \FF_p^2}{y \leq f(x)}$ to $\sett{(x,y) \in \FF_p^2}{y \geq -f(x)-1}$.
 Secondly, moving from $<$ to $\geq$ or from $\leq$ to $>$ is equivalent to taking the complement.
 Thus, it suffices to only focus on the sets defined using either $<$ or $\leq$.
 
 As observed in \Cref{Rmk:Parabola}, it is enough to fix a non-square $\alpha \in \FF_p$, and only consider the quadratic polynomials of the family
 \[
  \mf = \sett{X^2 + \gamma}{ \gamma \in \FF_p } \cup \sett{\alpha X^2 + \gamma}{ \gamma \in \FF_p }.
 \]
 
 Suppose that $p \geq 97$.
 Then we know from \Cref{thm:sym} that directions $(0)$ and $(\infty)$ are uniquely distinguished.
 Given the set $\ms = \sett{(x,y) \in \FF_p^2}{y < f(x)}$, with $f \in \mf$, we can reconstruct $f$.
 We can determine $\gamma$ using that $\nu(\gamma)$ is the only value which $\pr_{\ms,\infty}$ takes only once.
 Then the image of $\pr_{\ms,\infty}$ is either $(Q \cup \{0\}) + \gamma$ if $f(X) = X^2 + \gamma$ or $(N \cup \{0\}) + \gamma$ if $f(X) = \alpha X^2 + \gamma$.
 Hence, we find $2p$ pairwise non-isomorphic sets.
 Likewise, the $2p$ sets $\sett{(x,y) \in \FF_p^2}{y \leq f(x)}$, with $f \in \mf$ are pairwise non-isomorphic.
 The only remaining question is when the sets
 $\ms = \sett{(x,y) \in \FF_p^2}{y < f(x)}$ and $T = \sett{(x,y) \in \FF_p^2}{y \leq g(x)}$ with $f, g \in \mf$ can be isomorphic.
 We can again look at $\pr_{\ms,\infty}$ and $\pr_{T,\infty}$.
 It readily follows that $\ms$ and $T$ are isomorphic if and only if $f(X) = g(X)+1$ and $f(X)$ never takes the value $-1$.
 Since each quadratic polynomial $f$ takes $\frac{p+1}2$ different values, $\mf$ contains $2 \frac{p+1}2 = p+1$ polynomials $f$ that do take the value $-1$.
 We conclude that we find exactly $3p+1$ pairwise non-isomorphic sets.
\end{rmk}

\paragraph{Acknowledgements.} The first author is grateful to the second author for her hospitality during his visit to Budapest.
The first author was partially supported by Fonds Wetenschappelijk Onderzoek project 12A3Y25N and by a Fellowship of the Belgian American Educational Foundation.

\newcommand{\etalchar}[1]{$^{#1}$}

\end{document}